\documentclass{CSML}
\pdfoutput=1

\usepackage{lastpage}
\newcommand{\dOi}{13(3:18)2017}
\lmcsheading{}{1--\pageref{LastPage}}{}{}%
{Nov.~29, 2016}{Sep.~01, 2017}{}

\usepackage{hyperref}
\hypersetup{hidelinks}
\usepackage{amsmath}
\usepackage{amsfonts}
\usepackage{amsthm}
\usepackage{graphicx}
\usepackage{eucal}
\usepackage{amscd}
\usepackage[all,2cell]{xy}
\UseAllTwocells
\usepackage{amssymb}
\usepackage{mathrsfs}

\newdir{ |>}{{}*!/-3.5pt/@{|}*!/-8pt/:(1,-.2)@^{>}*!/-8pt/:(1,+.2)@_{>}}
\newtheoremstyle{defC}%
  {6pt}
  {6pt}
  {\normalfont}
  {}
  {\bfseries}
  {{\bfseries .}}
  {5pt plus 1pt minus 1pt}
  {\thmname{#1} \thmnumber{#2} \thmnote{\normalfont#3}}
\theoremstyle{defC}
\newtheorem{defiC}[thm]{Definition}

\newtheoremstyle{thmC}%
  {6pt}
  {6pt}
  {\itshape}
  {}
  {\bfseries}
  {{\bfseries .}}
  {5pt plus 1pt minus 1pt}
  {\thmname{#1} \thmnumber{#2} \thmnote{\normalfont#3}}

\newtheorem{thmC}[thm]{Theorem}
\newtheorem{propC}[thm]{Proposition}
  
\begin{document}

\title[Categorical-algebraic conditions in $S$-protomodular categories]{On some categorical-algebraic conditions in $S$-protomodular categories}

\author{Nelson Martins-Ferreira}
\address{ESTG, CDRSP, Instituto Polit\'{e}cnico de Leiria,
Leiria, Portugal}
\thanks{}
\email{martins.ferreira@ipleiria.pt}

\author{Andrea Montoli}
\address{Dipartimento di Matematica ``Federigo Enriques'', Universit\`{a} degli Studi di Milano, Italy}
\thanks{} \email{andrea.montoli@unimi.it}

\author{Manuela Sobral}
\address{CMUC and Departamento de
Matem\'atica, Universidade de Coimbra, 3001--501 Coimbra,
Portugal} \email{sobral@mat.uc.pt}

\keywords{$S$-protomodular categories, monoids, monoids with
operations, J\'{o}nsson-Tarski varieties, algebraic cartesian
closedness, algebraic coherence}

\subjclass[2010]{18G50, 03C05, 08C05, 18D35}

\dedicatory{Dedicated to Ji\v{r}\'\i~Ad\'amek}

\begin{abstract}
In the context of protomodular categories, several additional
conditions have been considered in order to obtain a closer
group-like behavior. Among them are locally algebraic cartesian
closedness and algebraic coherence. The recent notion of
$S$-protomodular category, whose main examples are the category of
monoids and, more generally, categories of monoids with operations
and J\'{o}nsson-Tarski varieties, raises a similar question: how
to get a description of $S$-protomodular categories with a strong
monoid-like behavior. In this paper we consider relative versions
of the conditions mentioned above, in order to exhibit the
parallelism with the ``absolute'' protomodular context and to
obtain a hierarchy among $S$-protomodular categories.
\end{abstract}

\maketitle

\section{Introduction}

Semi-abelian categories \cite{JMT} have been introduced in order
to give a categorical description of group-like algebraic
structures, as well as abelian categories describe abelian groups
and modules over commutative rings. However, the family of
semi-abelian categories revealed to be too large for this purpose,
since, together with algebraic structures like groups, rings, Lie
algebras, $\Omega$-groups in the sense of \cite{Higgins}, it
contains many other examples, like the duals of the categories of
pointed objects in a topos \cite{Bourn topos protomod} (in
particular, the dual of the
category of pointed sets is semi-abelian).

In order to get closer to group-like structures, several
conditions have been asked for a category, in addition to the
condition of being semi-abelian. A well studied one, which has
several important consequences in commutator theory and in the
description of internal structures, is the so-called ``Smith is
Huq'' condition \cite{MartinsVdL SH1}: two internal equivalence
relations on the same object centralize each other in the sense of
Smith-Pedicchio \cite{Smith, Pedicchio} if and only if their
normalizations commute in the sense of Huq \cite{Huq}. An example,
due to G. Janelidze, of a semi-abelian category (actually, a
semi-abelian variety in the sense of universal algebra) which
doesn't satisfy this condition is the category of digroups. As
shown in \cite{MantovaniMetere}, every category of \textit{groups
with operations} in the sense of Porter \cite{Porter} (see also
\cite{Orzech}, where the axioms for groups with operations have
been first considered, without giving a name to such structures)
is a semi-abelian category which satisfies the ``Smith is Huq''
condition. A characterization of the ``Smith is Huq'' condition in
terms of the so-called \textit{fibration of points} (see Section
\ref{preliminaries} for a description of this fibration) has been
given in \cite{BournMartinsVdL}: the change-of-base functors of
the fibration of points reflect the commutation of normal
subobjects.

Other conditions have been introduced more recently, by
constructing a parallelism with topos theory. We recall, among
them, locally algebraic cartesian closedness \cite{Gray,
BournGray}, fibrewise algebraic cartesian closedness
\cite{BournGray} and algebraic coherence \cite{CigoliGrayVdL}.
These conditions are obtained from the classical ones of locally
cartesian closedness, fibrewise cartesian closedness and coherence
by replacing, in their definition, the basic fibration with the
fibration of points. Actually, these additional conditions are
meaningful not only for semi-abelian categories, but also in more
general contexts, like pointed protomodular \cite{Bourn protomod}
categories. See Section \ref{preliminaries} for more details.
Every \textit{category of interest} in the sense of Orzech
\cite{Orzech} is algebraically coherent \cite[Theorem
4.15]{CigoliGrayVdL} and fibrewise algebraically cartesian closed
(this is a consequence of \cite[Theorem 6.27]{CigoliGrayVdL}).
Much less are the examples of locally algebraically cartesian
closed categories: the main ones are the categories of groups and
of Lie algebras over a commutative ring. Hence these conditions
create a hierarchy among semi-abelian and protomodular categories:
to a stronger condition corresponds a smaller family of examples,
closer to the main example, the category of groups.

An important property of semi-abelian categories is the fact that
internal actions (defined as in \cite{BJKinternalobjact}) are
equivalent to split extensions \cite[Theorem 3.4]{BJsemidir}. In
the category of groups, internal actions correspond to classical
group actions: an action of a group $B$ on a group $X$ is a group
homomorphism $B \to \text{Aut}(X)$. In the category of monoids,
which is not semi-abelian, the equivalence mentioned above does
not hold. Classical monoid actions are defined as monoid
homomorphisms $B \to \text{End}(X)$, where $\text{End}(X)$ is the
monoid of endomorphisms of the monoid $X$. Looking for a class of
split extensions that are equivalent to such actions in the
category of monoids, in \cite{MartinsMontoliSobral mon w op} we
identified a particular kind of split epimorphisms, that we called
\textit{Schreier split epimorphisms} (the name is inspired by the
work of Patchkoria on Schreier internal categories
\cite{Patchkoria}). A similar equivalence between actions and
Schreier split epimorphisms holds also for semirings, and, more
generally, for every category of \text{monoids with operations}
\cite{MartinsMontoliSobral mon w op}.

Further investigations on the class of Schreier split epimorphisms
in monoids, semirings and monoids with operations \cite{BMMS
SemForum, Schreier book} allowed to discover that this class
satisfies several properties that are typically satisfied by the
class of all split epimorphisms in a protomodular category, like
for example the Split Short Five Lemma, which is a key ingredient
in the definition of semi-abelian categories (in a pointed
finitely complete context, it is equivalent to protomodularity
\cite{Bourn protomod}). In order to describe this situation
categorically, the notion of pointed \textit{$S$-protomodular
category}, relatively to a suitable class $S$ of points (i.e.
split epimorphisms with a fixed section) has been introduced in
\cite{BMMS S-protomodular} (see Section \ref{S-protomodular
categories} for more details). In \cite{BMMS S-protomodular} it is
shown that $S$-protomodular categories satisfy, relatively to the
class $S$, many properties of protomodular categories. In
\cite{MartinsMontoli S-SH} it is proved that every
J\'{o}nsson-Tarski variety \cite{JonssonTarski} is
$S$-protomodular w.r.t. the class of Schreier split epimorphisms.
This is the case, in particular, for monoids with
operations, as already observed in \cite{BMMS S-protomodular}.

The aim of this paper is to study additional conditions on an
$S$-protomodular category, similarly to what has been done for
protomodular categories, in order to create a hierarchy among them
which allows to get closer to a categorical description of the
category of monoids, which is the central example of an
$S$-protomodular category. A relative version of the ``Smith is
Huq'' condition was already studied in \cite{MartinsMontoli S-SH}:
two $S$-equivalence relations (i.e. equivalence relations such
that the two projections, with the reflexivity morphism, form
points belonging to $S$, see \cite{BMMS S-protomodular}) on the
same object centralize each other if and only if their
normalizations commute. In \cite{MartinsMontoli S-SH} it was shown
that every category of monoids with operations satisfies this
relative version of the ``Smith is Huq'' condition. This already
permits to distinguish monoids with operations among
J\'{o}nsson-Tarski
varieties.

Now our aim is to consider relative versions of the other
additional conditions we mentioned, namely locally algebraic
cartesian closedness, fibrewise algebraic cartesian closedness and
algebraic coherence. In order to do that, we replace the fibration
of points with its subfibration of points belonging to the class
$S$, which is supposed to be stable under pullbacks (this
assumption is necessary to get a subfibration). We show that the
category of monoids is locally algebraically cartesian closed
w.r.t. the class $S$ of Schreier points, while this property fails
for semirings. Furthermore, the categories of monoids and
semirings are fibrewise algebraically cartesian closed and
algebraically coherent, relatively to $S$, while these two
properties fail, in general, for monoids with operations. Hence we
get the hierarchy we were looking for.

\section{Semi-abelian categories and additional conditions on
them} \label{preliminaries}

A semi-abelian category \cite{JMT} is a pointed, Barr-exact
\cite{BarrExact}, protomodular \cite{Bourn protomod} category with
finite coproducts. Every variety of $\Omega$-groups in the sense
of Higgins \cite{Higgins} is a semi-abelian category. More
generally, semi-abelian varieties of universal algebras have been
characterized in \cite{BJ semiab varieties}. Non-varietal examples
of semi-abelian categories are the category of compact Hausdorff
groups and the dual of every category of pointed objects
in a topos \cite{Bourn topos protomod}.

The condition of protomodularity can be expressed in terms of a
property of the so-called \textit{fibration of points}. A
\textit{point} in a category $\mathcal{C}$ is a $4$-tuple $(A, B,
f, s)$, where $f \colon A \to B$, $s \colon B \to A$ and $fs =
1_B$. In other terms, a point is a split epimorphism with a chosen
splitting. A morphism between a point $(A, B, f, s)$ and a point
$(A', B', f', s')$ is a pair $(g, h)$ of morphisms such that the
two reasonable squares in the following diagram commute:
\[ \xymatrix{ A \ar[d]_g \ar@<-2pt>[r]_f & B \ar[d]^h
\ar@<-2pt>[l]_s \\
A' \ar@<-2pt>[r]_{f'} & B', \ar@<-2pt>[l]_{s'} } \] i.e. $hf = f'
g$ and $gs = s' h$. There is a functor
\[ \text{cod} \colon Pt(\mathcal{C}) \to \mathcal{C} \]
which associates its codomain with every point: $\text{cod}(A, B,
f, s) = B$. If $\mathcal{C}$ has pullbacks, this functor is a
fibration, called the \textit{fibration of points}. A finitely
complete category $\mathcal{C}$ is protomodular if every
change-of-base functor of the fibration of points is conservative.
If $\mathcal{C}$ is pointed, protomodularity is equivalent
to the fact that the Split Short Five Lemma holds.

We now recall the definitions of the additional conditions we are
interested in.

\begin{defiC}[\cite{Gray, BournGray}] \label{lacc category}
A finitely complete category $\mathcal{C}$ is \emph{locally
algebraically cartesian closed}, or \emph{LACC}, if, for every
morphism $h \colon E \to B$, the corresponding change-of-base
functor of the fibration of points:
\[ h^* \colon Pt_B(\mathcal{C}) \to Pt_E(\mathcal{C}) \] has a
right adjoint.
\end{defiC}

As shown in \cite{Gray}, the categories of groups and of Lie
algebras over a commutative ring are LACC.

\begin{defiC}[\cite{BournGray}] \label{fwacc category}
A finitely complete category $\mathcal{C}$ is \emph{fibrewise
algebraically cartesian closed}, or \emph{FWACC}, if, for every
split epimorphism $h \colon E \to B$, the corresponding
change-of-base functor of the fibration of points has a right
adjoint.
\end{defiC}

Obviously every LACC category is FWACC. As shown in
\cite[Propositions 6.7 and 6.8]{Gray}, the category of
(non-unitary) rings is a
FWACC category which is not LACC.

We recall that, in a finitely complete category, a pair $(f, g)$
of morphisms with the same codomain is \emph{jointly strongly
epimorphic} if, whenever $f$ and $g$ factor through a monomorphism
$m$, $m$ is an isomorphism.

\begin{defiC}[\cite{CigoliGrayVdL}] \label{alg coherent category}
A finitely complete category $\mathcal{C}$ is \emph{algebraically
coherent} if, for every morphism $h \colon E \to B$, the
corresponding change-of-base functor of the fibration of points
preserves jointly strongly epimorphic pairs.
\end{defiC}

It is immediate to see that a finitely cocomplete LACC category is
algebraically coherent \cite[Theorem 4.5]{CigoliGrayVdL}. Any
semi-abelian algebraically coherent variety is FWACC \cite[Theorem
6.27]{CigoliGrayVdL}. Every \emph{category of interest} in the
sense of Orzech \cite{Orzech} is algebraically coherent
\cite[Theorem 4.15]{CigoliGrayVdL}, and hence FWACC. This is not
the case for every group with operations in the sense of Porter:
for example, the category of non-associative rings is not
algebraically coherent \cite[Examples 4.10]{CigoliGrayVdL}.

\section{$S$-protomodular categories} \label{S-protomodular
categories}

Given a finitely complete category $\mathcal{C}$, let $S$ be a
class of points in $\mathcal{C}$ which is stable under pullbacks
along any morphism. Denoting by $SPt(\mathcal{C})$ the full
subcategory of $Pt(\mathcal{C})$ whose objects are the points in
$S$, we obtain a subfibration of the fibration of points:
\[ S\text{-cod} \colon SPt(\mathcal{C}) \to \mathcal{C}. \]

A point $(A, B, f, s)$ is called a \emph{strong point} \cite{BMMS
S-protomodular} (or \emph{regular point}, as in
\cite{MartinsMontoliSobral ssfl normal}, in a regular context) if
the morphisms $k$ and $s$, where $k$ is a kernel of $f$, are
jointly strongly epimorphic.

\begin{defiC}[\hspace{-1pt}\cite{BMMS S-protomodular}] \label{S-protomodular category}
\hspace{-1pt}A pointed, finitely complete category $\mathcal{C}$ is said to be
\emph{$S$-protomodular} if:
\begin{enumerate}
\item every point in $S$ is a strong point;

\item $SPt(\mathcal{C})$ is closed under finite limits in
$Pt(\mathcal{C})$.
\end{enumerate}
\end{defiC}

The following result implies that the Split Short Five Lemma holds
for points in $S$ in an $S$-protomodular category:

\begin{thmC}[{\cite[Theorem 3.2]{BMMS S-protomodular}}] \label{change of base conservative}
In an $S$-protomodular category, every change-of-base functor of
the subfibration $S$-cod is conservative.
\end{thmC}

Every protomodular category $\mathcal{C}$ is $S$-protomodular for
the class $S$ of all points in $\mathcal{C}$. In order to give
other, more meaningful examples, we recall the following:

\begin{defiC}[\cite{JonssonTarski}] \label{Jonsson-Tarski varieties}
A variety in the sense of universal algebra is a
\emph{J\'{o}nsson-Tarski variety} if the corresponding theory
contains a unique constant $0$ and a binary operation $+$
satisfying the equalities $0 + x = x + 0 = x$  for any $x$.
\end{defiC}

\begin{defiC}[\cite{MartinsMontoliSobral mon w op, MartinsMontoli S-SH}] \label{Schreier split epi}
A split epimorphism $\xymatrix{ A \ar@<-2pt>[r]_f & B
\ar@<-2pt>[l]_s }$ in a J\'{o}nsson-Tarski variety is said to be a
\emph{Schreier split epimorphism} when, for any $a \in A$, there
exists a unique $\alpha$ in the kernel $\text{Ker}(f)$ of $f$ such
that $a = \alpha + sf(a)$.
\end{defiC}

Equivalently, a split epimorphism $(A, B, f, s)$ as above is a
Schreier split epimorphism if there exists a unique map $q_f
\colon A \to Ker(f)$ (which is not a homomorphism, in general)
such that $a = q_f(a) + sf(a)$ for all $a \in A$. This map $q_f$
is called the \emph{Schreier retraction} of the split epimorphism
$(A, B, f, s)$.

\begin{propC}[{\cite[Proposition 2.5]{MartinsMontoli S-SH}}]
If $\mathbb{C}$ is a J\'{o}nsson-Tarski variety and $S$ is the
class of Schreier split epimorphisms, then $\mathbb{C}$ is an
$S$-protomodular category.
\end{propC}

As shown in \cite{MartinsMontoliSobral mon w op} (see also Chapter
5 in \cite{Schreier book}), in the category of monoids Schreier
split epimorphisms are equivalent to monoid actions: an action of
a monoid $B$ on a monoid $X$ is a monoid homomorphism $\varphi
\colon B \to \text{End}(X)$, where $\text{End}(X)$ is the monoid
of endomorphisms of $X$. Given a Schreier split epimorphism
$\smash[b]{\xymatrix{ X \ar@<-2pt>[r]_k & A \ar@<-2pt>[r]_f
  \ar@<-2pt>@{.>}[l]_{q_f} & B \ar@<-2pt>[l]_s }}$
of monoids, the
corresponding action $\varphi$ is given by
\[ \varphi(b)(x )= q_f(s(b) + k(x)), \] where we use the additive
notation for the monoid operation; conversely, given an action of
$B$ on $X$, the corresponding Schreier split epimorphism is
obtained via a semidirect product construction. Actually, Schreier
split epimorphisms have been identified in order to get such an
equivalence. A similar equivalence with a suitable notion of
action holds in any category of \emph{monoids with operations}
\cite{MartinsMontoliSobral mon w op}, a family of varieties which
includes monoids, commutative monoids, semirings (i.e.~rings where
the additive structure is not necessarily a group, but just a
commutative monoid), join-semilattices with a bottom element,
distributive lattices
with a bottom element (or a top one).

The equivalence between actions and Schreier split epimorphisms
does not hold in any J\'{o}nsson-Tarski variety, that's why
originally in \cite{MartinsMontoliSobral mon w op} Schreier split
epimorphisms were only considered in monoids with operations. A
conceptual explanation of this phenomenon was given in
\cite{MartinsMontoli S-SH}: monoids with operations satisfy, with
respect to the class $S$ of Schreier split epimorphisms, a
relative version of the ``Smith is Huq'' condition: two
$S$-equivalence relations on the same object centralize each other
if and only if their normalizations commute. $S$-equivalence
relations (see \cite{BMMS S-protomodular}) are equivalence
relations such that the two projections, with the reflexivity
morphism, form points belonging to $S$. So, this condition allows
to distinguish monoids with operations from general
J\'{o}nsson-Tarski varieties. In order to get a more refined
classification, in the next section we consider relative versions
of the conditions on semi-abelian categories we recalled in
Section \ref{preliminaries}.

\section{Relative conditions on the fibration of points} \label{conditions for S-protomodular categories}

Throughout this section, let $\mathcal{C}$ be an $S$-protomodular
category, for a fixed class $S$ of points. By replacing the
fibration of points $\text{cod} \colon Pt(\mathcal{C}) \to
\mathcal{C}$ with its subfibration $S\text{-cod} \colon
SPt(\mathcal{C}) \to \mathcal{C}$ of points in $S$, we can
formulate relative versions of the conditions considered in
Section \ref{preliminaries}:

\begin{defi} \label{S-lacc category}
An $S$-protomodular category $\mathcal{C}$ is \emph{$S$-locally
algebraically cartesian closed}, or \emph{$S$-LACC}, if, for every
morphism $h \colon E \to B$, the corresponding change-of-base
functor of the fibration $S\text{-cod}$
\[ h^* \colon SPt_B(\mathcal{C}) \to SPt_E(\mathcal{C}) \] has a
right adjoint.
\end{defi}

\begin{defi} \label{S-fwacc category}
An $S$-protomodular category $\mathcal{C}$ is \emph{$S$-fibrewise
algebraically cartesian closed}, or \emph{$S$-FWACC}, if, for
every split epimorphism $h \colon E \to B$, the corresponding
change-of-base functor of the fibration $S\text{-cod}$ has a right
adjoint.
\end{defi}

\begin{defi} \label{S-alg coherent category}
An $S$-protomodular category $\mathcal{C}$ is
\emph{$S$-algebraically coherent} if, for every morphism $h \colon
E \to B$, the corresponding change-of-base functor of the
fibration $S\text{-cod}$ preserves jointly strongly epimorphic
pairs.
\end{defi}

Exactly as in the ``absolute'' case (i.e. when $S$ is the class of
all points), it is clear that every $S$-LACC category is $S$-FWACC
and that every finitely cocomplete $S$-LACC category is
$S$-algebraically coherent.

Our aim is to describe how these relative conditions allow to
distinguish ``poor'' $S$-protomodular categories, like general
J\'{o}nsson-Tarski varieties, from ``richer'' ones, closer to the
category of monoids. We start by showing that our key example, the
category of monoids, is $S$-LACC:

\begin{prop} \label{Mon is Schreier LACC}
The category $\mathsf{Mon}$ of monoids (and monoid homomorphisms)
is $S$-LACC, where $S$ is the class of Schreier split
epimorphisms.
\end{prop}

\begin{proof}
We already observed that a Schreier split epimorphism $\smash[t]{\xymatrix{
A \ar@<-2pt>[r]_f & B \ar@<-2pt>[l]_s }}$ with kernel $X$
corresponds to an action of $B$ on $X$, i.e. to a monoid
homomorphism $\varphi \colon B \to \text{End}(X)$. Moreover, such
actions can also be described as functors $F \colon B \to
\mathsf{Mon}$, where the monoid $B$ is seen as a category with
only one object $*_B$: given $\varphi$ as above, we define $F$ by
putting $F(*_B) = X$ and, for any $b \in B$, $F(b) = \varphi(b)$.
For every monoid $B$, there is then an equivalence of categories
\[ SPt_B(\mathsf{Mon}) \cong \mathsf{Mon}^B. \]
Given a monoid homomorphism $h \colon E \to B$, we have to prove
that the change-of-base functor
\[ h^* \colon SPt_B(\mathsf{Mon}) \to SPt_E(\mathsf{Mon}) \] has a
right adjoint. From the remarks above it follows that $h^*$ is
naturally isomorphic to the functor
\[ \mathsf{Mon}^h \colon \mathsf{Mon}^B \to \mathsf{Mon}^E. \] The
category $\mathsf{Mon}$ is complete, hence the right adjoint $R_h$
of $\mathsf{Mon}^h$ can be constructed by means of the right Kan
extension. Let us give a concrete description of $R_h$.

Given $F \colon E \to \mathsf{Mon}$ with $F(*_E) = M$, we have
that $R_h(F) \colon B \to \mathsf{Mon}$ is defined by
\[ R_h(*_B) = L(B, M) = \{ u \colon B \to M \ | \ e \cdot u(b) =
u(h(e)+b), \ \text{for every } b \in B, e \in E \}, \] where the
maps $u$ are not required to be monoid homomorphisms and $e \cdot
u(b)$ denotes the action of $e \in E$ on $u(b) \in M$; the
definition of $R_h$ on morphisms is given by
\[ (R_h(b_0)(u))(b) = u(b+b_0). \]
This last equality describes the action of $B$ on $L(B, M)$. The
counit $\varepsilon$ of the adjunction $\mathsf{Mon}^h \dashv R_h$
has components $\varepsilon_F \colon L(B, M) \to M$ given by
evaluation of $u \in L(B, M)$ at the neutral element of $B$:
$\varepsilon_F(u) = u(0_B)$. Let us check the universality of this
counit. Given a functor $G \colon B \to \mathsf{Mon}$, with
$G(*_B) = S$, and a natural transformation $\beta \colon Gh \to
F$, there is a unique natural transformation $\gamma \colon G \to
R_h(F)$ such that $\varepsilon \gamma h = \beta$. The component of
$\gamma$ at the unique object $*_B$ of $B$ is the map
$\gamma_{*_B} \colon S \to L(B, M)$ defined by
\begin{equation}
  \gamma_{*_B}(s)(b) = \beta(b \cdot s), \ \text{for all } s \in
S, b \in B.
\tag*{\qEd}
\end{equation}
\def\popQED{}
\end{proof}

We observe that, with the notation of the previous proof, if $h
\colon E \to B$ is a surjective monoid homomorphism, and if we
choose a set-theoretical section $s \colon B \to E$, then the
description of $L(B, M)$ can be simplified. In fact, every $u \in
L(B, M)$ is completely determined by its value at $0_B$:
\[ u(b) = u(b +0_B) = u(hs(b) +0_B) = s(b) \cdot u(0_B). \] Hence
$L(B, M)$ is isomorphic to the following submonoid of $M$:
\[ \{ m \in M \ | \ e s(b) \cdot m = s(h(e)+b) \cdot m, \ \text{for every } b \in B, e \in E \}. \]

\medskip

Being $S$-LACC, the category of monoids is also $S$-FWACC and
$S$-algebraically coherent. The category $\mathsf{SRng}$ of
semirings is not $S$-LACC with respect to the class $S$ of
Schreier split epimorphisms: if $B$ is a ring, $B \neq 0$, then
every split epimorphism with codomain $B$ is a Schreier one
\cite[Proposition 6.1.6]{Schreier book}: $SPt_B(\mathsf{SRng}) =
Pt_B(\mathsf{SRng})$. Then Proposition 6.7 in \cite{Gray} implies
that the functor
\[ \text{Ker}_B \colon SPt_B(\mathsf{SRng}) \to \mathsf{SRng} =
SPt_0(\mathsf{SRng}), \] which is the change-of-base functor
determined by the unique morphism $0 \to B$, does not have a right
adjoint. However, $\mathsf{SRng}$ is $S$-FWACC. Actually, we can
prove something slightly stronger:

\begin{prop} \label{SRng is Schreier fwacc}
If $h \colon E \to B$ is a regular epimorphism (i.e. a surjective
homomorphism) in the category $\mathsf{SRng}$ of semirings (and
semiring homomorphisms), then the change-of-base functor
\[ h^* \colon SPt_B(\mathsf{SRng}) \to SPt_E(\mathsf{SRng}) \]
of the fibration of Schreier points has a right adjoint.
\end{prop}

\begin{proof}
Similarly to what happens for monoids, Schreier split epimorphisms
of semirings correspond to actions \cite{MartinsMontoliSobral mon
w op}. An action of a semiring $B$ on a semiring $X$ is a pair
$\varphi = (\varphi_l, \varphi_r)$ of functions
\[ \varphi_l \colon B \times X \to X, \qquad \varphi_r \colon X
\times B \to X, \] whose images are simply denoted by $b \cdot x$
and $x \cdot b$, respectively, such that the following conditions
are satisfied for all $b, b_1, b_2 \in B$ and all $x, x_1, x_2 \in
X$:
\begin{enumerate}
\item $0 \cdot x = x \cdot 0 = 0 \cdot b = b \cdot 0 = 0$;

\item $b \cdot (x_1 + x_2) = b \cdot x_1 + b \cdot x_2, \quad $
$(x_1 + x_2) \cdot b = x_1 \cdot b + x_2 \cdot b$;

\item $(b_1 + b_2) \cdot x = b_1 \cdot x + b_2 \cdot x, \quad $ $x
\cdot (b_1 + b_2) = x \cdot b_1 + x \cdot b_2$;

\item $b \cdot (x_1 x_2) = (b \cdot x_1) x_2, \quad $ $(x_1 x_2)
\cdot b = x_1 (x_2 \cdot b)$;

\item $(b_1 b_2) \cdot x = b_1 \cdot (b_2 \cdot x), \quad $ $x
\cdot (b_1 b_2) = (x \cdot b_1) \cdot b_2$;

\item $x_1 (b \cdot x_2) = (x_1 \cdot b) x_2, \quad $ $(b_1 \cdot
x) \cdot b_2 = b_1 \cdot (x \cdot b_2)$.
\end{enumerate}
\vspace{2mm}
Given a Schreier split epimorphism $\smash{\xymatrix{ X \ar@<-2pt>[r]_k &
A \ar@<-2pt>[r]_f \ar@<-2pt>@{.>}[l]_{q_f} & B \ar@<-2pt>[l]_s }}$
of semirings, the corresponding action is given by
\[ b \cdot x = q_f(s(b) k(x)), \qquad x \cdot b = q_f(k(x) s(b)); \]
conversely, given an action of $B$ on $X$, the corresponding
Schreier split epimorphism is obtained via a semidirect product
construction (see \cite{MartinsMontoliSobral mon w op} for more
details). If we denote by $B$-$Act$ the category whose objects are
pairs $(X, \varphi)$, with $X \in \mathsf{SRng}$ and $\varphi$ an
action of $B$ on $X$, and whose morphisms are equivariant
homomorphisms, we get an equivalence of categories
$SPt_B(\mathsf{SRng}) \cong B$-$Act$. Unlike the case of monoids,
a $B$-action can't be represented as a functor into
$\mathsf{SRng}$. Still, given a surjective morphism $h \colon E
\to B$ in $\mathsf{SRng}$, in order to prove that the
change-of-base functor
\[ h^* \colon SPt_B(\mathsf{SRng}) \to SPt_E(\mathsf{SRng}) \] has a
right adjoint, we can consider the equivalent functor
\[ h^* \colon B\text{-}Act \to E\text{-}Act \] which sends $(X, \varphi)$ to
$(X, \psi)$, where the action $\psi$ is defined by
\[ e \cdot x = h(e) \cdot x, \qquad x \cdot e = x \cdot h(e). \]
The right adjoint $R_h \colon E$-$Act \to B$-$Act$ of $h^*$ is
defined as follows. Given $(X, \psi) \in E$-$Act$, we define
\[ R_h(X) = \{ x \in X \ | \ e_1 \cdot x = e_2 \cdot x, \ x \cdot
e_1 = x \cdot e_2 \ \text{for all } e_1, e_2 \in E \ \text{such
that } h(e_1) = h(e_2) \}. \] It is immediate to see that $R_h(X)$
is a submonoid of $X$. The action $\varphi$ of $B$ on $R_h(X)$ is
given by
\[ b \cdot x = e \cdot x, \quad x \cdot b = x \cdot e \quad
\text{for all } e \in E \ \text{such that } h(e) = b. \] Of course
this is well-defined thanks to the definition of $R_h(X)$. Given a
morphism $g \colon X \to Y$ in $E$-$Act$, $R_h(g)$ is just its
restriction to $R_h(X)$: it takes values in $R_h(Y)$ because of
the equivariance of $g$. It is straightforward to check that $R_h$
is the right adjoint to $h^*$.
\end{proof}

\begin{cor}
The category $\mathsf{SRng}$ is $S$-FWACC, where $S$ is the class
of Schreier split epimorphisms.
\end{cor}

The category of semirings is not only $S$-FWACC, but also
$S$-algebraically coherent:

\begin{prop} \label{SRng is Schreier algebraically coherent}
The category $\mathsf{SRng}$ is $S$-algebraically coherent, where
$S$ is the class of Schreier split epimorphisms.
\end{prop}

\begin{proof}
Since all the change-of-base functors of the fibration of Schreier
points in $\mathsf{SRng}$ are conservative, as we recalled in
Section \ref{S-protomodular categories}, and they obviously
preserve monomorphisms, we can use Lemmas 3.9 and 3.11 in
\cite{CigoliGrayVdL} to conclude that, in order to prove that
$\mathsf{SRng}$ is $S$-algebraically coherent, it suffices to show
that, for every semiring $B$, the kernel functor
\[ \text{Ker}_B \colon SPt_B(\mathsf{SRng}) \to \mathsf{SRng} \]
preserves jointly strongly epimorphic pairs. Accordingly, consider
the following diagram, where $f$ and $g$ are morphisms in
$SPt_B(\mathsf{SRng})$:
\[ \xymatrix{ H \ar[r]^{f_{|_H}} \ar@{ |>->}[d] & K \ar@{ |>->}[d]
& L \ar[l]_{g_{|_L}} \ar@{ |>->}[d] \\
A \ar[r]^f \ar@<-2pt>[d]_{p'} & D \ar@<-2pt>[d]_p & C \ar[l]_g
\ar@<-2pt>[d]_{p''} \\
B \ar@{=}[r] \ar@<-2pt>[u]_{s'} & B \ar@{=}[r] \ar@<-2pt>[u]_s &
B, \ar@<-2pt>[u]_{s''} } \] and suppose that $f$ and $g$ are
jointly strongly epimorphic (in $SPt_B(\mathsf{SRng})$); we have
to prove that their restrictions to $H$ and $L$, respectively, are
jointly strongly epimorphic in $\mathsf{SRng}$. It is not
difficult to see that $f$ and $g$ are jointly strongly epimorphic
in $SPt_B(\mathsf{SRng})$ if and only if they have the same
property in $\mathsf{SRng}$. This happens if and only if $D$, as a
semiring, is generated by the images $f(A)$ and $g(C)$. This means
that every $d \in D$ may be written as sums of elements of the
following 4 forms:
\[ f(a_1) g(c_1) f(a_2) g(c_2) \ldots f(a_n) g(c_n), \qquad f(a_1) g(c_1) f(a_2) g(c_2)
\ldots f(a_n), \]
\[ g(c_1) f(a_2) g(c_2) \ldots f(a_n) g(c_n),
\qquad g(c_1) f(a_2) g(c_2) \ldots f(a_n). \] In particular, every
element $k \in K$ has this property. We have to show that $k$ can
be written as a sum of products as above, but only using elements
of $f(H)$ and $g(L)$.

We start by considering the simplest case: suppose that $k = f(a)
g(c)$ for some $a \in A$, $c \in C$. Since $(p', s')$ and $(p'',
s'')$ are Schreier split epimorphisms, there are (unique) $h \in
H$, $l \in L$ such that $a = h + s'p'(a)$ and $c = l + s''p''(c)$.
Then
\begin{align} \label{formula 1}
k &= f(a)g(c) = f(h + s'p'(a)) g(l + s''p''(c)) =
  \\
&= f(h)g(l) + f(h) gs''p''(c) + fs'p'(a) g(l) + fs'p'(a)
gs''p''(c) = \tag*{}
\\
& = f(h)g(l) + f(h) sp''(c) + sp'(a) g(l) + sp'(a)
  sp''(c). \tag*{}
\end{align}
  We have that
\[ f(h) sp''(c) = f(h) fs'p''(c) = f(h \, s'p''(c)) \] and
\[ p'(h \, s'p''(c)) = pf(h \, s'p''(c)) = pf(h) psp''(c) = 0 p''(c) = 0, \]
hence $h \, s'p''(c) \in H$. Similarly, we show that $sp'(a) g(l)
= g(s''p'(a) \, l)$, with $s''p'(a) \, l \in L$. From
\eqref{formula 1} we get
\begin{align*}
0 = p(k) &= p(f(h)g(l) + f(h) sp''(c) + sp'(a) g(l) + sp'(a) sp''(c)) = \\
&= p(sp'(a) sp''(c)) = ps(p'(a) p''(c));
\end{align*}
since $ps = 1_B$, we
get that $p'(a) p''(c) = 0$ and hence $s(p'(a) p''(c)) = sp'(a)
sp''(c) = 0$. This means that $k$ can be written as a sum of
products of elements in $f(H)$ and $g(L)$. The case in which $k =
g(c)f(a)$, for some $c \in C$ and $a \in A$, is similar.

If $k = f(a_1) g(c) f(a_2)$ for some $a_1, a_2 \in A$ and $c \in
C$, then
\begin{align*}
k &= f(h_1 + s'p'(a_1)) g(l + s''p''(c)) f(h_2 + s'p'(a_2)) = \\
&= (f(h_1) + sp'(a_1)) (g(l) + sp''(c)) (f(h_2) + sp'(a_2))
\end{align*}
for suitable $h_1, h_2 \in H$, $l \in L$. Making the calculations in
the last expression using the distributivity law, we get a sum in
which every summand, except $sp'(a_1) sp''(c) sp'(a_2)$, contains
an element of $f(H)$ or of $g(L)$, and then all these summands
belong either to $f(H)$, to $g(L)$ or are made of products of
elements of $f(H)$ and $g(L)$; let's consider, for instance, one
of them:
\[ f(h_1)sp''(c)sp'(a_2) = f(h_1 \, s'p''(c) s'p'(a_2)), \] and $h_1 \, s'p''(c)
s'p'(a_2)\in H$, since $p'(h_1 \, s'p''(c) s'p'(a_2)) = 0$. So, it
only remains to consider the summand $sp'(a_1) sp''(c) sp'(a_2)$;
but, for the same reasons as in the case $k = f(a)g(c)$, this is
equal to $0$. So, if $k = f(a_1) g(c) f(a_2)$, then $k$ is a sum
of products of elements of $f(H)$ and $g(L)$. All the other cases
are dealt analogously.
\end{proof}

It is not true that every category of monoids with operations is
$S$-algebraically coherent w.r.t. the class $S$ of Schreier split
epimorphisms. This is not the case even for groups with
operations, as shown in Examples $4.10$ in \cite{CigoliGrayVdL}.
Hence the additional conditions we considered in this section gave
us the hierarchy among $S$-protomodular categories we were looking
for. When $S$ is the class of Schreier split epimorphisms, this
hierarchy is summarized by the following table:

\bigskip

\begin{center}
  \def\arraystretch{1.2}
\begin{tabular}{cc}
\hline 
    \sffamily Property & \sffamily True in \\
    \hline \hline
$S$-LACC & $\mathsf{Mon}$ \\
$S$-algebraically coherent & $\mathsf{Mon}$, $\mathsf{SRng}$ \\
$S$-FWACC & $\mathsf{Mon}$, $\mathsf{SRng}$ \\
$S$-``Smith is Huq'' & categories of monoids with operations \\
$S$-protomodularity & J\'{o}nsson-Tarski varieties  \\
    \hline
\end{tabular}
\end{center}

\section*{Acknowledgements}

This work was partially supported by the Programma per Giovani
Ricercatori ``Rita Levi Montalcini'', funded by the Italian
government through MIUR.

This work was partially supported by the Centre for Mathematics of
the University of Coimbra - UID/MAT/00324/2013, and by the
project CompositeSteering (18024), by ESTG and CDRSP from the
Polytechnic of Leiria - UID/Multi/04044/2013. Both projects
UID/MAT/00324/2013 and UID/Multi/04044/2013 are funded by the
Portuguese Government through FCT/MCTES and co-funded by the
European Regional Development Fund through the Partnership
Agreement PT2020.


\begin{thebibliography}{99}

\bibitem{BarrExact} M. Barr, \emph{Exact categories}, in Lecture Notes in Mathematics, vol. 236 (1971), Springer-Verlag, 1-120.

\bibitem{BJKinternalobjact} F. Borceux, G. Janelidze, G.M. Kelly, \emph{Internal object actions},
Comment. Math. Univ. Carolin. 46 n.2 (2005), 235-255.

\bibitem{Bourn protomod} D. Bourn, \emph{Normalization equivalence, kernel equivalence and affine
categories}, in Lecture Notes in Mathematics, vol. 1488 (1991),
Springer-Verlag 43-62.

\bibitem{Bourn topos protomod} D. Bourn, \emph{Protomodular aspect of the dual of a topos}, Adv. Math. 187 (2004), no. 1, 240-255.

\bibitem{BournGray} D. Bourn, J.R.A. Gray, \emph{Aspects of
algebraic exponentiation}, Bull. Belg. Math. Soc. Simon Stevin 19
(2012), 823-846.

\bibitem{BJsemidir} D. Bourn, G. Janelidze, \emph{Protomodularity, descent, and semidirect
products}, Theory Appl. Categ. 4 (1998), No. 2, 37-46.

\bibitem{BJ semiab varieties} D. Bourn, G. Janelidze, \emph{Characterization of protomodular varieties of universal
algebras}, Theory Appl. Categ. 11 (2003), No. 6, 143-147.

\bibitem{Schreier book} D. Bourn, N. Martins-Ferreira, A. Montoli,
M. Sobral, \emph{Schreier split epimorphisms in monoids and in
semirings}, Textos de Matem\'{a}tica (S\'{e}rie B), Departamento
de Matem\'{a}tica da Universidade de Coimbra, vol. 45 (2013).

\bibitem{BMMS SemForum} D. Bourn, N. Martins-Ferreira, A. Montoli,
M. Sobral, \emph{Schreier split epimorphisms between monoids},
Semigroup Forum 88 (2014), 739-752.

\bibitem{BMMS S-protomodular} D. Bourn, N. Martins-Ferreira, A. Montoli,
M. Sobral, \emph{Monoids and pointed $S$-protomodular categories},
Homology, Homotopy and Applications 18 n.1 (2016), 151-172.

\bibitem{BournMartinsVdL} D. Bourn, N. Martins-Ferreira, T. Van
der Linden, \emph{A characterisation of the ``Smith is Huq''
condition in the pointed Mal'tsev setting}, Cah. Topol. G\'{e}om.
Diff\'{e}r. Cat\'{e}gor. \textbf{LIV} (2013), n. 4, 243-263.

\bibitem{CigoliGrayVdL} A.S. Cigoli, J.R.A. Gray, T. Van der
Linden, \emph{Algebraically coherent categories}, Theory Appl.
Categ. 30 (2015), 1864-1905.

\bibitem{Gray} J.R.A. Gray, \emph{Algebraic exponentiation in
general categories}, Appl. Categ. Structures 20 (2012), 543-567.

\bibitem{Higgins} P.J. Higgins, \emph{Groups with multiple operators}, Proc. Lond.
Math. Soc. (3) 6 (1956), no. 3, 366-416.

\bibitem{Huq} S.A. Huq, \emph{Commutator, nilpotency and
solvability in categories}, Q. J. Math. 19 (1968) n.2, 363-389.

\bibitem{JMT} G. Janelidze, L. M\'{a}rki, W. Tholen, \emph{Semi-abelian
categories}, J. Pure Appl. Algebra 168 (2002), no. 2-3, 367-386.

\bibitem{JonssonTarski} B. J\'{o}nsson, A. Tarski, \emph{Direct
decompositions of finite algebraic systems}, Notre Dame
Mathematical Lectures, Notre Dame, Indiana (1947).

\bibitem{MantovaniMetere} S. Mantovani, G. Metere, \emph{Internal
crossed modules and Peiffer condition}, Theory Appl. Categ. 23
(2010), 113-135.

\bibitem{MartinsMontoli S-SH} N. Martins-Ferreira, A. Montoli,
\emph{On the ``Smith is Huq" condition in $S$-protomodular
categories}, Appl. Categ. Structures 25 (2017), 59-75.

\bibitem{MartinsMontoliSobral mon w op} N. Martins-Ferreira, A. Montoli, M.
Sobral, \emph{Semidirect products and crossed modules in monoids
with operations}, J. Pure Appl. Algebra 217 (2013), 334-347.

\bibitem{MartinsMontoliSobral ssfl normal} N. Martins-Ferreira, A. Montoli,
M. Sobral, \emph{Semidirect products and Split Short Five Lemma in
normal categories}, Appl. Categ. Structures 22 (2014), 687-697.

\bibitem{MartinsVdL SH1} N. Martins-Ferreira, T. Van der Linden,
\emph{A note on the ``Smith is Huq'' condition}, Appl. Categ.
Structures 20 (2012), 175-187.

\bibitem{Orzech} G. Orzech, \emph{Obstruction theory in algebraic categories I}, J. Pure Appl. Algebra 2 (1972),
287-314.

\bibitem{Patchkoria} A. Patchkoria, \emph{Crossed semimodules and Schreier internal
categories in the category of monoids}, Georgian Math. Journal 5
(1998), n.6, 575-581.

\bibitem{Pedicchio} M.C. Pedicchio, \emph{A categorical approach
to commutator theory}, J. Algebra 177 (1995), 647-657.

\bibitem{Porter} T. Porter, \emph{Extensions, crossed modules and
internal categories in categories of groups with operations},
Proceedings of the Edinburgh Math. Society 30 (1987), 373-381.

\bibitem{Smith} J.D.H. Smith, \emph{Mal'cev varieties}, Lecture
Notes in Mathematics, vol. 554 (1976), Springer.

\end{thebibliography}
\end{document}